\documentclass[11pt]{article}

\usepackage{enumerate}
\usepackage{amsthm,amsmath,amssymb}
\usepackage{graphicx}
\usepackage[colorlinks=true,citecolor=black,linkcolor=black,urlcolor=blue]{hyperref}
\usepackage[english]{babel}
\usepackage{amsfonts}
\usepackage{epsfig, subfigure}
\usepackage{amscd,latexsym}
\usepackage{url}
\usepackage{color}
\usepackage{authblk}
\usepackage{subfigure}

\newcommand{\modq}{\text{ mod }4}

\newcommand\blfootnote[1]{%
	\begingroup
	\renewcommand\thefootnote{}\footnote{#1}%
	\addtocounter{footnote}{-1}%
	\endgroup
}

\theoremstyle{plain}
\newtheorem{theorem}{Theorem}
\newtheorem{lemma}[theorem]{Lemma}
\newtheorem{corollary}[theorem]{Corollary}
\newtheorem{proposition}[theorem]{Proposition}

{\itshape}{\rmfamily}

\newtheorem{conjecture}[theorem]{Conjecture}

\theoremstyle{plain}

\date{\today}

\title{The Equidistant Dimension of  Graphs}

\author{A. Gonz\'{a}lez\thanks{Department of Didactics of Mathematics, Universidad de Sevilla.
E-mail address:
gonzalezh@us.es. \\Partially supported by the "VI Plan Propio de Investigación y Transferencia" of the Universidad de Sevilla (Spain) and by the Research Group in Mathematics Education FQM-226 of the Junta de Andalucía (Spain).},
C. Hernando\thanks{Departament de Matem\`{a}tiques, Universitat Polit\`{e}cnica de Catalunya.\\
E-mail addresses:
 ${\rm \{}$carmen.hernando,merce.mora${\rm \}}$@upc.edu.
 Partially supported by projects
 Gen.Cat. DGR2017SGR1336 and PID2019-104129GB-I00/AEI/10.13039/501100011033  of the Spanish Ministry of Science and Innovation.},
M. Mora$^{\dag}$\thanks{Also supported by European project H2020-734922-CONNECT.}
}

\begin{document}
\maketitle

	\blfootnote{\begin{minipage}[l]{0.3\textwidth}
		\includegraphics[trim=10cm 6cm 10cm 5cm,clip,scale=0.15]{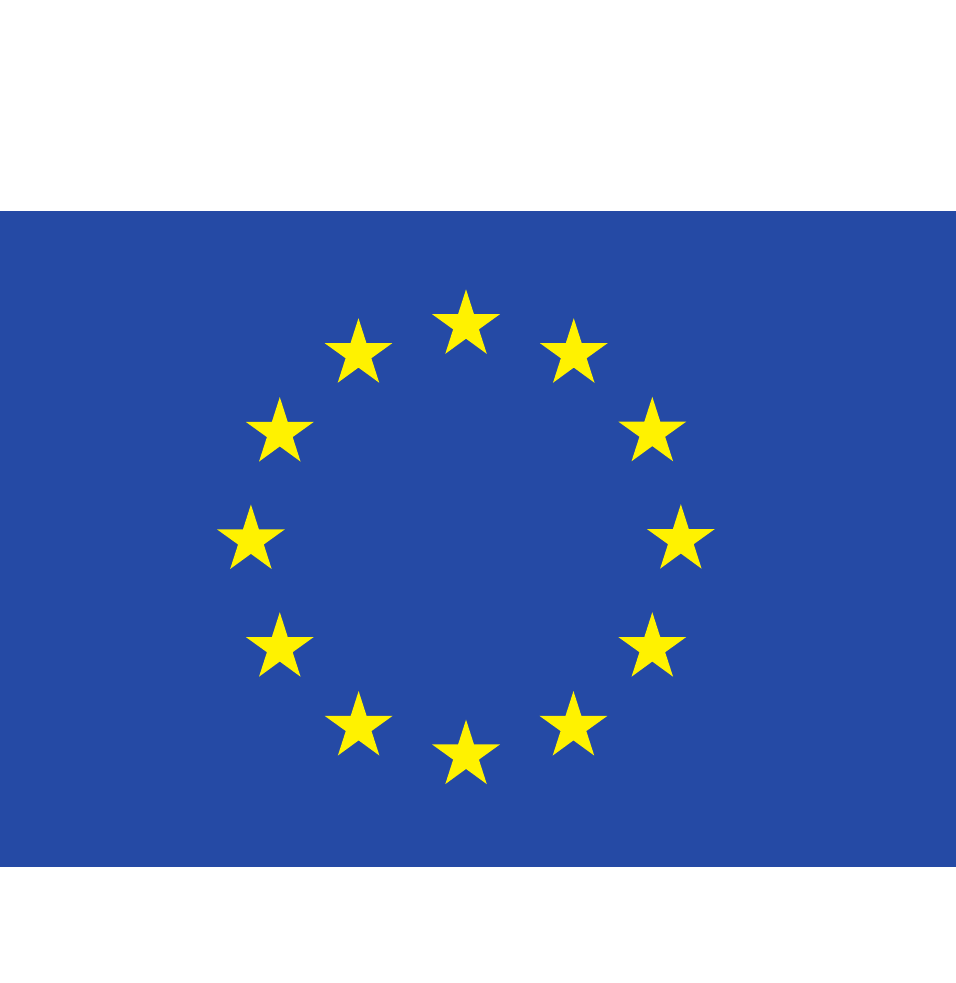}
	\end{minipage}  \hspace{-2cm} \begin{minipage}[l][1cm]{0.7\textwidth}
		This project has received funding from the European Union's Horizon 2020 research and innovation programme under the Marie Sk\l{}odowska-Curie grant agreement No 734922.
\end{minipage}}

\begin{abstract}
A subset $S$ of vertices of a connected graph $G$ is a {\em distance-equalizer set} if for every two distinct
vertices $x, y \in V (G) \setminus S$ there is a vertex $w \in S$ such that the distances from $x$ and $y$ to
$w$ are the same. The {\em equidistant dimension} of $G$ is the minimum cardinality of a distance-equalizer set of
G. This paper is devoted to introduce this parameter and explore its properties and
 applications to other mathematical problems, not necessarily in the context of graph theory. 
Concretely, we first establish some bounds concerning the order, the maximum degree, the clique number, and the independence number,
and characterize all graphs attaining some extremal values. We then study the equidistant dimension of several families of graphs
(complete and complete multipartite graphs, bistars, paths, cycles, and Johnson graphs), proving that, in the case of paths and cycles,
this parameter is related with 3-AP-free sets. Subsequently,
 we show the usefulness of distance-equalizer sets for constructing doubly resolving sets.
\end{abstract}


\section{Introduction}\label{intro}
The notion of resolving set, also known as locating set,
	was introduced by Slater \cite{slater} and, independently, by Harary and Melter \cite{harary}. This concept arises in diverse areas, including location problems in  networks  of different nature (see  \cite{our}).
	For example, in order to locate a failure in a computer network modeled as a graph, we are interested in a subset
 of vertices $S$ such that every vertex of the underlying graph might be uniquely determined by its vector of distances to
 the vertices of $S$. Such a set is called a {\em resolving set} of the graph, and the {\em metric dimension} of that graph is the minimum cardinality of a resolving set.

Resolving sets and several related sets, such as identifying codes, locating-dominating sets or watching systems, have been widely studied during the last decades (see \cite{b,ourinf,CGHHMMT21,lobstein}), as well as doubly resolving sets, a type of subset of vertices more restrictive than resolving sets  with multiple applications in different areas~\cite{our,kratica0,cw,ghm,j,doubly, kratica2}. 
However, many recent papers~\cite{c,cs,cs2,anonimo,yero,z,z2} have turned their attention precisely in the opposite direction to resolvability, thus trying to study anonymization problems in networks instead of location aspects.
 For instance, the need to ensure privacy and anonymity in social networks makes necessary to develop graph tools
such as the concepts of  antiresolving set and  metric antidimension,  introduced by
Trujillo-Rasua and Yero~\cite{yero}. Indeed, a subset of vertices $A$ is a {\it 2-antiresolving set} if, for every vertex $v\notin A$, there exists another different vertex $w\notin A$ such that $v$ and $w$ have the same vector of distances to the vertices of $A$; the {\it 2-metric antidimension} of a graph is the minimum cardinality among all its 2-antiresolving sets. With the same spirit, this paper introduces new graph concepts that can also be applied to anonymization problems in networks: \emph{distance-equalizer set} and \emph{equidistant dimension}. Furthermore, we shall see that these concepts have concrete applications in mathematical problems, such as obtaining new bounds on the size of doubly resolving sets of graphs, as well as a new formulation in terms of graphs of a classical problem of number theory.

The paper is organized as follows.
In Section~\ref{basics}, we define distance-equalizer sets and the equidistant dimension, and show bounds in terms of other graph parameters: order, diameter, maximum degree, independence number, and clique number.
Section~\ref{extremal} is devoted to characterize all graphs attaining some extremal values of the equidistant dimension.
In Section~\ref{families} we study this parameter for some families of graphs: complete and complete multipartite graphs, bistars,
 paths, cycles, and Johnson graphs. For the particular cases of paths and cycles, we show that this parameter is related with 3-AP-free sets.
  In Section~\ref{doubly}, we obtain bounds for general graphs and trees on the minimum cardinality of doubly resolving sets in terms
  of the equidistant dimension. 
Finally, we present some conclusions and open problems in Section~\ref{op}.

All graphs considered in this paper are connected, undirected, simple and finite. The
 vertex set and the edge set of a graph $G$ are denoted by  $V(G)$ and $E(G)$, respectively;
 its \emph{order} is $ |V(G)|$. For any vertex $v\in V(G)$, its \emph{open neighborhood} is the set $N(v) = \{w \in V(G)  : vw \in E(G)\}$ and its \emph{closed neighborhood} is $ N[v] = N(v)\cup\{v\}$;
 its \emph{degree}, denoted by
$deg(v)$, is defined as the cardinality of $N(v)$. If $deg(v) = 1$ then we
 say that $v$ is a \emph{leaf}, in which case the only vertex adjacent to $v$ is called its \emph{support vertex}; when $\deg (v)=|V(G)|-1$, we say that $v$ is \emph{universal}.
The \emph{maximum degree} of $G$ is $\Delta(G) = \textrm{max} \ \{\textrm{deg}(v) : v \in V(G)\}$ and
its \emph{minimum degree} is
$\delta(G) = \textrm{min} \ \{\textrm{deg}(v) : v \in V(G)\}$.
The \emph{distance} between two vertices $v,w \in V(G)$ is denoted by $d(v,w)$, and
the {\em diameter} of $G$ is $D(G) =\textrm{ max} \{d(v,w) : v,w \in V(G)\}$.
A \emph{clique} is a subset of pairwise adjacent vertices and the \emph{clique number} of $G$, denoted by $\omega (G)$,  is the maximum cardinality of a clique of $G$; an \emph{independent set}  is a subset of pairwise non-adjacent vertices and the {\em independence number} of $G$, denoted by $\omega (G)$, is the maximum cardinality of an independent set of $G$.
For undefined terms we refer the reader to \cite{west}.

For every integer $n\ge 1$, let $[n]=\{1,2,\dots ,n\}$. We denote by $P_n$ the \emph{path} of order $n$ with vertex set $[n]$  and edges $ij$ with j=i+1 and $i\in[n-1]$, and by $C_n$ the \emph{cycle} of order $n$, $n\ge 3$, with vertex set  $[n]$ and the same edge set as $P_n$ together with the edge $1n$. Also, for every $r,s\geq 3$, we denote by $K_{1,r}$ the {\em star} on $r+1$ vertices with vertex set $\{v\} \cup [r]$,  vertex $v$ being called the {\em center} of the star, and edge set $\{vi: i\in [r]\}$;
 a {\it bistar}, denoted by $K_2(r,s)$, is a graph obtained by joining the centers of two stars $K_{1,r-1}$ and $K_{1,s-1}$.

\section{Distance-equalizer sets and equidistant dimension}\label{basics}

Let $x,y,w$ be vertices of a graph $G$.
We say that $w$ is {\it equidistant  } from $x$ and $y$ if $d(x,w)=d(y,w)$.
A subset $S$ of vertices is called a {\it distance-equalizer set} for $G$ if for every two distinct vertices $x,y\in V(G)\setminus S$ there exists a vertex $w \in S$ equidistant from $x$ and $y$; the
{\it equidistant dimension} of $G$, denoted by $eqdim(G)$, is the minimum cardinality of a distance-equalizer set of $G$.
For example, if $v$ is a universal vertex of a graph $G$, then $S=\{v\}$ is a minimum distance-equalizer set of $G$, and so $eqdim (G)=1$.
Also, a distance-equalizer set of $P_8$ is shown in Figure \ref{path}, and it can be easily checked that $P_8$ has no distance-equalizer set of size at most $4$, so $eqdim(P_8)=5$.
\begin{figure}[!t]
	\begin{center}
		\includegraphics[width=0.4\textwidth]{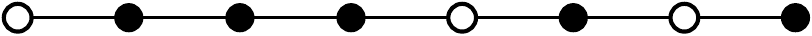}
		\caption{Black vertices form a distance-equalizer set of minimum size for $P_8$.}
		\label{path}
	\end{center}
\end{figure}

The following results are immediate but make it easier to prove subsequent results.

\begin{lemma}\label{hoja-soporte}
	Let $G$ be a graph. 	
		If $S$ is a distance-equalizer set of $G$ and $v$ is a support vertex of $G$, then $S$ contains $v$ or all leaves adjacent to $v$.
		Consequently, $$eqdim(G)\geq \mid \{v\in V(G) \, : \, v \,\, \textrm{is a support vertex} \}\mid.$$
	 \end{lemma}

\begin{proof}
	No vertex is equidistant from a leaf and its support vertex, since every path from a leaf to any other vertex goes through its support vertex.  Hence, if $v$ is not in $S$, then all leaves hanging from $v$ must be in $S$.	
\end{proof}

Recall that a graph $G$ is  {\em bipartite} whenever $V(G)$ can be partitioned into two independent sets, say $A,B$, which
are called its {\em partite sets}.

\begin{proposition} \label{propbipartitos}
	Let $G$ be a bipartite graph with partite sets $A$ and $B$. If $S$ is a distance-equalizer set of $G$, then $A\subseteq S$ or  $B\subseteq S$.
		Consequently, $eqdim(G)\geq min\{\vert A \vert, \vert B \vert \}.$
\end{proposition}

\begin{proof}
		The distance between two vertices in the same partite set is even, while the distance between vertices of different partite sets is odd. Hence, there is no vertex equidistant from two vertices belonging to different partite sets. Therefore, $A\subseteq S$ or  $B\subseteq S$.
\end{proof}

If $G$ is a graph of order $n$, with $n\geq 2$, then any set of vertices of cardinality $n-1$ is obviously a distance-equalizer set. Hence, $n-1$ is an immediate upper bound on the equidistant dimension of nontrivial graphs.
We next prove some upper bounds involving classical graph parameters.

\begin{proposition}\label{generalbounds}
For every  graph $G$  of order $n\geq 2$, the following statements hold.
\begin{enumerate}
  \item[i)] $eqdim(G)\leq n-\Delta(G)$ and the bound is tight whenever $\Delta (G) \ge n/2$;
  \item[ii)] $eqdim(G)\leq n- \omega (G) +1$;
   \item[iii)] $eqdim(G)\leq \frac {n(D(G)-1) +1}{D(G)}$;
  \item[iv)] $eqdim(G)\leq n- \alpha(G)  +1$, whenever $D(G)=2$.

\end{enumerate}
\end{proposition}

\begin{proof}
\begin{enumerate}
  \item[i)] Let $v$ be a vertex of degree $\Delta(G)$. It is easy to see that the set $S=V(G)\setminus N(v)$ is a distance-equalizer set of cardinality $n-\Delta(G)$, and so
  $eqdim(G)\leq n-\Delta(G)$.
  To prove tightness, let $H_{a,b}$ with $a\geq 1$ and $0\leq b< a$ be the graph with vertex set $\{v, v_1, \ldots, v_a,u_1,\ldots,u_b\}$ and edge set $\{vv_i: i\in[a]\}\cup\{v_iu_i: i\in[b]\}$. This graph has order $a+b+1$ and maximum degree $a$,
  so $eqdim (H_{a,b})\leq  b +1 $
  as we have just seen. Moreover, this graph is bipartite with partite sets $A=\{v,u_1,\ldots, u_b\}$ and $B=\{v_1,\ldots,v_a\}$, and so
  by Proposition~\ref{propbipartitos} we have that  $eqdim (H_{a,b})\geq min\{\vert A \vert, \vert B \vert \} = b +1 $ since $b<a$.
  Hence, $H_{a,b}$ is a graph of order $a+b+1$ and maximum degree $a$, with  $a\geq \frac{a+b+1}{2}$ and $eqdim (H_{a,b})= b +1 $, showing that the given the bound is tight.
  \item[ii)] Let $W$ be a clique of $G$ of cardinality $\omega(G)$ and let $w\in W$. The set $S=(V(G)\setminus W)\cup \{w\}$ is a distance-equalizer set since every vertex not in $S$ is adjacent to $w$, and so $eqdim(G)\leq |S|= n- \omega(G) +1$.
    \item[iii)]  Let $v$ be a vertex of $G$ for which there exits another vertex at distance $D(G)$. For every $i\in [1,D(G)]$, let $V_i$ be the set of vertices at distance $i$ from $v$, and observe that all vertices in $V_i$ are equidistant from $v$. Also, there must exist a set $V_{i_0}$, with $1\le i_0\le D$, having at least $\frac{n-1}{D(G)}$ vertices. Therefore, $V(G)\setminus V_{i_0}$ is a distance-equalizer set, and consequently $eqdim(G)\le n-\frac{n-1}{D(G)}=\frac{n(D(G)-1)+1}{D(G)}$.
  \item[iv)]  Let $W$ be an independent set of order $\alpha(G)$ and let $w\in W$. If $D(G)=2$, then $d(w,v)=2$ for any other vertex $v \in W$. Thus, $S=(V(G)\setminus W)\cup \{w\}$ is a distance-equalizer set and we obtain $eqdim(G)\leq |S|=n- \alpha(G)  +1$.
\end{enumerate}
\end{proof}

  The bound given in Proposition~\ref{generalbounds}(i) is not tight for all values of $\Delta(G) $ and $n$, for example when  $\Delta(G)=2$ and $n\ge 7$. Indeed, the only graphs satisfying $\Delta(G) = 2$ are  paths and cycles and, as it will be seen below, the equidistant dimension of paths and cycles of order $n\ge 7$ is at most $n-3$.



\section{Extremal values}\label{extremal}

In this section we characterize all nontrivial graphs achieving extremal values for the equidistant dimension, concretely, graphs $G$ of order $n\geq 2$ such that $eqdim (G)\in \{ 1,2,n-2,n-1\}$. We also derive a Nordhaus-Gaddum type bound for the equidistant dimension.

\begin{theorem}\label{extremalinferior} For every graph $G$ of order $n\geq 2$, the following statements hold.
	\begin{enumerate}[i)]
		\item $eqdim(G)=1$ if and only if $\Delta(G)=n-1$;
		\item $eqdim(G)=2$ if and only if $\Delta(G)=n-2$.
	\end{enumerate}
\end{theorem}

\begin{proof}
	\begin{enumerate}[$i)$]
		\item If $\Delta(G)=n-1$ then $eqdim(G)\le n-(n-1)=1$ by Proposition~\ref{generalbounds}(i), and so $eqdim(G)=1$. 		
		Conversely, if $eqdim (G)=1$ then there exists a vertex $v$ such that $S=\{ v \}$ is a distance-equalizer set of $G$. We claim that $v$ has degree $n-1$. Indeed, suppose on the contrary that there is a vertex $u$ not adjacent to $v$. Then, since there is at least a vertex $w$ adjacent to $v$, we have  $d(v,w)=1\not=d(v,u)$. Hence, $\{ v\}$ is not a distance-equalizer set of $G$, a contradiction.
		Therefore, $G$ has maximum degree $n-1$.
	
		\item
		If $G$ has maximum degree $n-2$, then $eqdim(G)\le n-(n-2)=2$ by Proposition~\ref{generalbounds}(i), and $eqdim(G)\not=1$ by item (i). Hence, $eqdim(G)=2$.	Conversely, suppose that $ eqdim(G)=2$ and let $S=\{u,v\}$ be a distance-equalizer set of $G$.
	We first prove  that
	$N(u)\setminus \{ v\}\subseteq N(v)\setminus \{u\}$ or $N(v)\setminus \{ u\}\subseteq N(u)\setminus \{v\}$.
	Suppose on the contrary that $N(u)\setminus \{ v\}\nsubseteq N(v)\setminus \{u\}$ and $N(v)\setminus \{ u\}\nsubseteq N(u)\setminus \{v\}$.  Then, there exist vertices $x,y$ such that
	$x\in N(u)\setminus \{ v\}$ and $x\notin N(v)\setminus \{u\}$,
	$y\in N(v)\setminus \{ u\}$ and $y\notin N(u)\setminus \{v\}$. Thus,
	 vertices $x,y$ verify that $d(x,u)=1\neq d(y,u)$ and
	$d(y,v)=1\neq d(x,v)$, contradicting that  $S$ is a distance-equalizer set.    	
	Without loss of generality, we assume $N(v)\setminus \{ u\}\subseteq N(u)\setminus \{v\}$.

	Next, we prove that $V(G)=N[u]\cup \{v\}$.
	Suppose on the contrary  that there is a vertex $z\notin N[u]\cup \{v\}$.
	Thus, $d(z,u)\ge 2$ and $d(z,v)\ge 2$.
	If $N(v)\setminus \{u\}$ is nonempty, then for any vertex  $x\in N(v)\setminus \{u\}$ we have $d(x,u)=1\neq d(z,u)\ge 2$ and $d(x,v)=1\neq d(z,v)\ge 2$, contradicting that $S$ is a distance-equalizer set.
	Otherwise $N(v)\setminus \{u\}$ is empty, and so $v$ is a leaf with $u$ as its support vertex. Thus, for any $y\in N(u)\setminus \{v\}$, we have $d(y,u)=1\neq d(z,u)\ge 2$ and $d(y,v)=2\not= d(z,v)=d(z,u)+1\ge 3$, contradicting again that $S$ is distance-equalizer.
	
	Finally, we have that $\Delta(G) \not= n-1$ by the preceding item. Hence, $v\notin N(u)$  and $eqdim (G)=deg(u)=n-2$.
	\end{enumerate}
\end{proof}

\begin{theorem}\label{extremalsuperior}
	For any graph $G$ of order $n$, the following statements hold.
	\begin{enumerate}[i)]
		\item If $n\ge 2$ then $eqdim(G)=n-1$ if and only if $G$ is a path of order $2$.
		\item If $n\ge 3$ then $eqdim(G)=n-2$ if and only if $G \in \{ P_3,P_4,P_5, P_6, C_3,C_4, C_5\}$.
	\end{enumerate}
\end{theorem}

\begin{proof}
	\begin{enumerate}[i)]
		\item It is obvious that $eqdim(P_2)=1$. Conversely, if $G$ is a graph with $eqdim(G)=n-1$, then $\Delta(G)=1$  by Proposition~\ref{generalbounds}(i), and the only connected graph with maximum degree equal to 1 is the path of order 2.
		
		\item A straightforward computation shows that the graphs $P_3$, $P_4$, $P_5$, $P_6$, $C_3$, $C_4$ and $C_5$ have equidistant dimension equal to the order minus 2.
		Conversely, if $G$ is a graph with $eqdim(G)=n-2$, then $\Delta(G)\leq 2$   by Proposition~\ref{generalbounds}(i).
As we have seen above, the path of order 2 is the only connected graph with maximum degree 1. Hence,
		$\Delta(G)=2$, that is, $G$ is a path or a cycle of order at least 3.
		It is easy to see that in both cases the set $[n]\setminus \{ 1,3,7\}$ is a distance-equalizer set whenever $n\ge 7$, and so $eqdim (G)\le n -3$ but $eqdim(G)=n-2$. Therefore, $3\leq n\leq 6$ and consequently $G\in \{ P_3,P_4,P_5, P_6, C_3,C_4, C_5\}$ since $eqdim (C_6)=3\not= 6-2$.  		
	\end{enumerate}	
	
\end{proof}

\begin{corollary}
	If $G$ is a graph of order $n\geq 7$, then
	$1\le eqdim(G) \le n-3.$
\end{corollary}


Now, we provide a Nordhaus-Gaddum type bound on the equidistant dimension. Nordhaus-Gaddum type inequalities establish bounds on the sum of a parameter for a graph and its complement.
Recall that the \emph{complement} of a graph $G$, denoted by $\overline{G}$, is the graph on the same vertices as $G$ and two vertices are adjacent in $\overline{G}$ if and only if they are not adjacent in $G$.
Also, a graph $G$ is {\it doubly connected } if both $G$ and $\overline{G}$ are connected.
Note that nontrivial doubly connected graphs have order at least $4$.

The following result is a direct consequence of Proposition \ref{generalbounds}(i).

\begin{proposition}\label{doublyconnected}
	If $G$ is a doubly connected graph, then $eqdim ( \overline{G} )\leq \delta (G)+1$.
\end{proposition}

\begin{theorem} If $G$ is a doubly connected graph of order $n\ge 4$, then $4\leq eqdim(G)+eqdim(\overline{G})\leq n+1.$
	{Moreover, these bounds are tight.}
\end{theorem}

\begin{proof}
First observe that a graph $G$ satisfying
		$eqdim(G)=1$ is not doubly connected. Indeed, in such
		a case, by Theorem~\ref{extremalinferior}(i), it contains a
		universal vertex $v$ that is an isolated vertex in $\overline{G}$. Hence,
		$eqdim (G)\ge 2$ and $eqdim (\overline{G})\ge 2$, whenever $G$ is doubly connected, and the lower bound follows.
		The family of  bistars $G=K_2(2,n-2)$, $n\ge 4$, provides examples of graphs
		attaining the lower bound for every $n\ge 4$. As we will see in Theorem~\ref{bipartitoscompletos}(iii) below, these graphs satisfy $eqdim(G)=2$,
		and it is easy to check that $eqdim(\overline{G})=2$.
		\medskip
		
		The upper bound is a direct consequence of Propositions \ref{generalbounds}(i) and \ref{doublyconnected}, because
		\begin{align*}
		eqdim(G)+eqdim(\overline{G})\leq n- \Delta (G)+ \delta (G)+1 \le n+1.
		\end{align*}

		The  cycle  $C_5$ attains the upper bound, since $\overline{C_5}=C_5$ and it is easy to check that $eqdim(C_5)=3$.
\end{proof}


\section{Equidistant dimension of some families of graphs}\label{families}

In this section we study the equidistant dimension of some families of graphs,
concretely of complete, complete bipartite and complete multipartite graphs, bistars, paths, cycles, and Johnson graphs.

\subsection{Complete graphs, complete multipartite graphs and bistars}\label{KKKKK}

Recall that, for every positive integer $n$, the
 \emph{complete graph} $K_n$ is the graph of order $n$ in which every pair of  vertices is connected by an edge. Also,
the {\em complete bipartite graph} $K_{r,s}$,
  with $r,s$  positive integers, is the bipartite graph with partite sets  $A,B$  such that  $|A|=r$ and $|B|=s$, and
  edge set given by all  pairs $vu$ with $v\in A$ and $u\in B$.
  More generally, a {\it complete $p$-partite} graph, denoted by $K_{n_1,\dots, n_p}$, is a graph with set of vertices $A_1\cup \dots \cup A_p$
such that $A_1, \dots ,A_p$, which are called its {\it partite sets}, are pairwise disjoint, verify $|A_i|=n_i\ge 1$, and two vertices are adjacent if and only if they belong to $A_i$ and $A_j$, respectively, with $i\not= j$. Note that complete bipartite graphs are thus 2-partite graphs.

\begin{theorem}\label{bipartitoscompletos} Let $n,r,s,p,n_1,\dots, n_p$ be positive integers such that $n\ge 2$, $s\ge r$, $p\ge 3$
and $n_p\ge \dots \ge n_1\ge 1$. Then, the following statements hold.
	\begin{enumerate}[i)]
		\item  $ eqdim(K_{n})=1$;
		\item  $ eqdim(K_{r,s})=r$;
		\item  $ eqdim(K_2 (r,s))=r$;
		\item $ eqdim(K_{n_1,\dots, n_p})=\min \{n_1,3\}$
	\end{enumerate}
\end{theorem}

\begin{proof}
	\begin{enumerate}[$i)$]
		
		\item It is a direct consequence of Theorem~\ref{extremalinferior}(i).
		
		\item  By Proposition~\ref{propbipartitos}, $eqdim(K_{r,s})\geq r$. Since both partite sets of $K_{r,s}$ are distance-equalizer sets, we have $eqdim(K_{r,s})=r$.

		\item Note that $K_2 (r,s)$ is a bipartite graph with partite sets of cardinality $r$ and $s$. Hence, by Proposition~\ref{propbipartitos},
		$eqdim(K_2{(r,s)})\geq r$. Moreover, if $A$ is the partite set of cardinality $r$, then $A$ is a distance-equalizer set because it contains a vertex at distance 1 from every vertex not in $A$.
		
		\item If $n_1=1$ then $K_{n_1,\dots, n_p}$ has a universal vertex and $eqdim(K_{n_1,\dots, n_p})=1$, by Theorem~\ref{extremalinferior}(i).
		 If $n_1=2$ then $K_{n_1,\dots, n_p}$ has maximum degree equal to the order minus 2 and $eqdim(K_{n_1,\dots, n_p})=2$, by Theorem~\ref{extremalinferior}(ii).
		 Otherwise $n_1\ge 3$, and so the maximum degree of $K_{n_1,\dots, n_p}$ is at most the order minus 3 and $eqdim(K_{n_1,\dots, n_p})\ge 3$, by Theorem~\ref{extremalinferior}. Moreover, tt is very easy to verify that any set consisting of 3 vertices from different partite sets is distance-equalizer. Thus, we conclude that $eqdim(K_{n_1,\dots, n_p})=3.$
	\end{enumerate}
\end{proof}

\subsection{Paths}\label{paths}

We next show that distance-equalizer sets and the equidistant dimension of paths are  related with   3-AP-free sets and the function $r(n)$
 introduced by Erd\"{o}s and Tur\'{a}n \cite{erdos}.
  A subset $S\subseteq [n]$ is 3-{\it AP-free} if $a+c \neq 2b$, for every distinct terms $a,b,c \in S$;
 the largest cardinality of a 3-AP-free subset of $[n]$ is denoted by $r(n)$.

We begin by introducing some preliminary results.
A subset of $[n]$ is called \emph{even-sum} if all its elements have the same parity.

\begin{proposition}\label{apfreecamino}
Let $S\subseteq [n]$ for some integer $n$. Then,
	$S$ is a distance-equalizer set of $P_n$ if and only if $[n] \setminus S$ is a 3-AP-free even-sum set.
\end{proposition}

\begin{proof}
	Let us denote by $A$ the set of  vertices of $P_n$ labelled with even numbers  and by $B$ the set of  vertices labelled with odd numbers.
	(Note that $P_n$ is a bipartite graph and $A,B$ are its partite sets).
	Also, let $S$ be a distance-equalizer set of $P_n$ with $|S|=r$.
	By Proposition~\ref{propbipartitos}, either $A \subseteq S$ or $B  \subseteq S$.
	Thus, if  $T= [n] \setminus S=\{t_1, \dots, t_{n-r}\}$, then $t_1,\dots ,t_{n-r}$ have the same parity, that is, $T$ is  a even-sum set.
	Moreover, $(t_i+t_j)/2$  is the only vertex of $P_n$ equidistant from  $t_i$ and $t_j$. Hence $(t_i+t_j)/2 \in S$, that is, $(t_i+ t_j)/2 \notin T$. Then, $T$ is a 3-{\it AP-free} set. 	
	Conversely, suppose that $T=\{t_1, \dots, t_{n-r}\}$ is a 3-AP-free even-sum set. Then, for all pair of vertices $t_i$, $t_j$ of $T$, we have $(t_i+t_j)/2 \in [n] \setminus T$. Hence, $S=[n] \setminus T$ is a distance-equalizer set of $P_n$.
\end{proof}

\begin{corollary} For every positive integer $n$, it holds that
	
	$$ eqdim(P_n)=n-\textrm{max} \{ \, |T| \, : \,T \textrm{ is a 3-AP-free even-sum subset of }[n] \}.$$
	
\end{corollary}

\begin{proposition}\cite{cockaine}\label{ap-free} Let $ k_1,..., k_r, n$ be different positive integers. Then, one of the sets
	$\{2k_1-1,2k_2 -1,...,  2k_r -1\}$ or $\{2k_1,2k_2,...,  2k_r\}$ is a 3-AP-free even-sum set of $[n]$ if and only if
	$ \{k_1,..., k_r\}$ is a 3-AP-free subset of $\Big[ \, \lceil n/2 \rceil \, \Big]$.
\end{proposition}

The equidistant dimension of a path is derived from the results above.

\begin{theorem} \label{camino}
	For every positive integer $n$, it holds that
	
	$$eqdim(P_n)=n-r \left( \Big\lceil {{\frac{n}{2}}} \Big\rceil \right).$$
	
\end{theorem}

Hence, obtaining the equidistant dimension of paths amounts to computing the function
 $r(n)$,  which  has been widely
 studied~\cite{b2,*bloom16,*bloom19,*bourgain,*elkin,*gasarch,*gowers,hb,roth,*sanders,sharma}.
In fact, many papers are devoted to obtain the values of $r(n)$ in some specific cases ($n\leq 23$ and $n=41$~\cite{erdos};
$n\leq 27$ and $41\leq n \leq 43$~\cite{sharma};  $n\le 123$~\cite{sequences}), which
 allows us to compute $eqdim(P_n)$  in all those cases  (see Table~\ref{tab:eqdim_caminosciclos} for $n\leq 20$ and $n=50$).
Also, other works \cite{b2,sequences,roth} provide bounds on $r(n)$ that are useful to approach $eqdim(P_n)$, such as
$$ n^{1-c/\sqrt{\log{n}}}< r(n) < \frac{cn}{\log{\log {n}}}.$$

Besides its relationship with the function $r(n)$, the equidistant dimension of paths is also related with a
 problem concerning covering squares of a chessboard by queens proposed by Cockayne and Hedetniemi~\cite{cockaine}.
 Indeed, the authors are interested in determining the minimum number of queens needed to be placed on the major diagonal
 of a chessboard in order to reach all the remaining squares with a single chess movement.
More formally, a subset $K \subseteq [n]$ is a \emph{diagonal dominating set} if its $|K|$ queens placed
in position $\{(k,k) \, |  \, k \in K\}$ on the black major diagonal of an $n\times n$ chessboard cover the entire board;
the minimum cardinality of a diagonal dominating set is denoted by $diag(n)$.
 It is proved in~\cite{cockaine} that diagonal dominating sets are precisely the complements of 3-AP-free  even-sum sets, which combined
 with Proposition~\ref{apfreecamino}  leads us to see that
the distance-equalizer sets of $P_n$ are the diagonal dominating sets in $[n]$, and consequently $eqdim(P_n)=diag(n)$.

Finally, we do not know the exact value of the equidistant dimension of trees.
However, in this family of graphs, it looks that paths are those graphs needing more
vertices to construct a distance-equalizer set. Indeed, it is easily seen that,
for every pair of vertices of a path, there is at most one equidistant vertex. Hence,
we believe that the following conjecture holds true.

\noindent
\begin{conjecture}\label{conjtrees}
\textit{If $T$ is a tree of order $n$, then $eqdim(T)\leq eqdim(P_n)$.}
\end{conjecture}
\vspace{0.3cm}

\subsection{Cycles}\label{cycles}

In this section, the equidistant dimension of cycles of even order is completely determined,
while for cycles of odd order, lower and upper bounds in terms of $r(n)$ are given.

\begin{theorem}
	For every positive integer $n \geq 3$, the following statements hold.
	\begin{enumerate}
		\item[i)]
		$eqdim(C_n)=\left\{
		\begin{array}{lll}
		\frac{n}{2} , & \hbox{for n even,} \ \ n\not\equiv 0 \modq; \\
		\\
		\frac{3n}{4}-1, & \hbox{for n even,} \ \ n\equiv 0 \modq. \\
		
	\end{array}
	\right.$
	\item[ii)] $\frac{n-1}{2}\leq   eqdim(C_n) \leq n-r\left ( \Big \lceil\frac{n+1}{4} \Big \rceil  \right), \, \, \hbox{for n odd.}$
	
\end{enumerate}
\end{theorem}

\begin{proof}
Throughout this proof, 	
 for every $i,j\in [n]$, we use the expression $\frac{i+j+n}{2}$ to represent the
only integer in $[n]$ modulo $n$ whenever $\frac{i+j+n}{2}$ is an integer.	Thus,
	for every pair of vertices $i,j$ of $C_n$, the vertices
equidistant from  them are $\frac{i+j}{2}$ and $\frac{i+j+n}{2}$,
whenever these values are integers. Hence, there is exactly one vertex equidistant from $i$ and $j$,
when $n$ is odd; there is no equidistant vertex from $i$ and $j$, whenever $n$ is even and $i$, $j$ h
ave distinct parity; and there are exactly two vertices equidistant from $i$ and $j$, if $n$ even and $i$, $j$ have
the same parity. Moreover, in this last case, the vertices equidistant from $i$ and $j$ are antipodal.

\begin{enumerate}[$i)$]
	\item Let $n$ be an even integer, and let us
	denote by $A$ (resp., $B$) the set of  vertices of $C_n$ labelled with odd (resp., even) numbers.
	As $n$ is even, $C_n$ is a bipartite graph and, by Proposition~\ref{propbipartitos}, for every distance-equalizer set $S$, either $A\subseteq S$  or $B\subseteq S$. Hence, $|S|\geq n/2$.
	We distinguish two subcases.

	\begin{enumerate}
		
	\item {Case $n\not= 0 \modq$}. We claim that $A$ is a distance-equalizer set (see an example in Figure~\ref{cycles}(a)).
	Indeed, for every $i,j\in [n]\setminus A$, the numbers $\frac{i+j}{2}$ and $\frac{i+j+n}{2}$ are
integers of different parity, because $n$ is even but $n\not= 4k$. Thus, either $\frac{i+j}{2}$ or $\frac{i+j+n}{2}$ belongs
 to $A$. Hence, $A$ is a distance-equalizer set and $eqdim{(C_n)}=n/2$.
	
	\item {Case $n=0 \modq$}. Let $S$ be a distance-equalizer set, and let us assume, relabeling the vertices
if necessary, that $A\subseteq S$. Thus, $[n]\setminus S\subseteq B$.
	First, we suppose that there is a pair of antipodal vertices in $[n]\setminus S\subseteq B$.
	We can assume without loss of generality that these vertices are $n/2$ and $n$.
 For every $i\in \{2,4,\dots ,n/2-2\}$, the only vertices equidistant from $i$ and $n-i$ are $n/2$ and $n$.
	Since $n/2$ and $n$ are not in $S$, we derive that
	at least one of the vertices $i$ or $n-i$ must be in $S$, for every $i\in \{2,4,\dots ,n/2-2\}$.
	Therefore, besides the vertices from $A$, the set $S$ contains at least $\frac{n/2-2}2$ vertices from $B$. Therefore,
	
	$$eqdim(C_n)\geq \frac{n}{2}+\frac{n/2-2}{2}=\frac{3n}{4}-1.$$

It is straightforward that the same bound holds if we suppose that there is no pair of antipodal vertices in $[n]\setminus S\subseteq B$.

	Now, we consider the set $S=[n]\setminus \{2,4,6,\dots ,n/2+2\}$ of size $|S|=\frac{3n}{4}-1$.
	It is easy to check that every pair of vertices not in $S$ has a vertex in $\{n/2+3,n/2+4,\dots ,n,1\}\subseteq S$ equidistant from them.
 Thus, we conclude that $S$ is a distance-equalizer set of minimum cardinality (see an example in Figure~\ref{cycles}(b)), and so
     $eqdim(C_n)= \frac{3n}{4}-1$.	
\end{enumerate}

\begin{figure}[!t]
	\centering
	\includegraphics[width=0.8\textwidth]{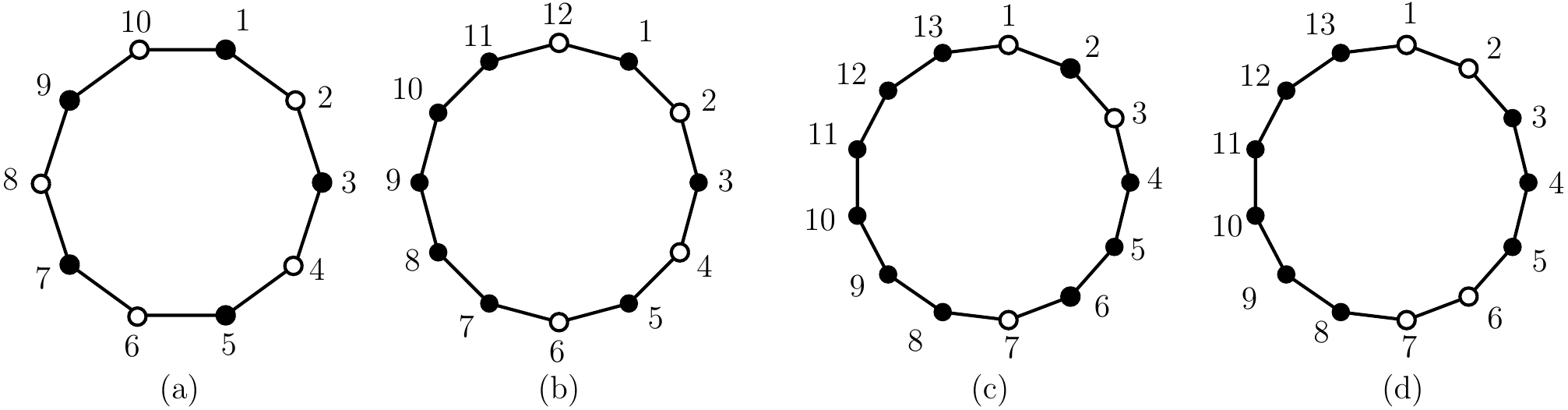}\hspace{2mm}
	\caption{The set of black vertices form a distance-equalizer set of the cycles (a) $C_{10}$, (b) $C_{12}$, and (c,d) $C_{13}$. In cases (a), (b) and (d), the given distance-equalizer set has minimum cardinality.}
	\label{cycles}
\end{figure}

	\item Let $n$ be an odd integer,  and 	
	let $S$ be a distance-equalizer set of minimum size.
	Since   $|S|\le n-1$, by Proposition~\ref{generalbounds}(i), we can assume without loss of generality that $n\notin S$.
 As $n$ is an odd integer, $n$ is the only vertex of $C_n$ equidistant from each pair of vertices $i,n-i$,
 with $i\in \{1,\dots ,(n-1)/2\}$, and so at least one of them must be in $S$.
	Therefore, $eqdim (C_n)\geq (n-1)/2$.
	
To prove the upper bound, let $S_1=\{ i : (n+1)/2 < i \le n\}$ and consider a
distance-equalizer set $S_2$ of $P_{(n+1)/2}$. We claim that $S=S_1 \cup S_2$ is a
distance-equalizer set of $C_n$  (see an example in Figure~\ref{cycles}(c)). Indeed, any two vertices $i,j$ not in $S$ belong to $[(n+1)/2] $, and there
is a vertex in $S_2$ equidistant from then, since $S_2$ is a distance-equalizer set of $P_{(n+1)/2}$. Hence,

	$$eqdim(C_n)\leq |S_1|+|S_2|=\frac{n-1}{2}+eqdim(P_{\frac{n+1}{2}})
	=n-r\left ( \Big \lceil\frac{n+1}{4} \Big \rceil  \right).$$
	
	\end{enumerate}
\end{proof}

Note that the distance-equalizer set constructed in the proof of the preceding theorem for odd cycles is not necessarily of minimum cardinality. In Figure~\ref{cycles}(c,d), the distance-equalizer set described in the proof of the theorem and a distance-equalizer set of minimum cardinality for $C_{13}$ are shown.

In Table~\ref{tab:eqdim_caminosciclos}
the values of the equidistant dimension of $C_n$ for $n \leq 20$ and $n=50$ are given.
Note that some of these values have been obtained with computer.

\begin{table}[!t]
	\center{
		\begin{tabular}{p{1.5cm}p{.27cm}p{.27cm}p{.27cm}p{.27cm}p{.27cm}p{.27cm}p{.27cm}p{.27cm}p{.27cm}p{.27cm}p{.27cm}p{.33cm}p{.33cm}p{.33cm}p{.33cm}p{.33cm}p{.33cm}p{.35cm}p{.33cm}}
			&&&&&&&&&&&&&&&&&&& \\
			\ \  \  \ \  \      n  & 3 & 4 & 5 & 6 & 7 & 8 & 9 & 10 & 11 & 12 & 13 & 14 & 15 & 16 & 17 & 18 & 19 & 20 & 50 \\
			\hline
			$r(\lceil n/2 \rceil)$       & 2&2&2&2&3&3&4&4&4&4&4&4&4&4&5&5&5&5&10\\		
			\hline
			$eqdim(P_n)$  & 1 & 2 & 3 & 4 & 4 & 5 & 5 & 6  & 7  & 8  & 9  & 10 & 11 & 12 & 12 & 13 & 14 & 15 & 40 \\
			\hline
			$eqdim(C_n)$  & 1 & 2 & 3 & 3 & 4 & 5 & 5 & 5  & 7  & 8  & 9  & 7  & 11 & 11 & 12 & 9  & 13 & 14 & 25 \\
	\end{tabular}}
	\caption{Equidistant dimension of some paths and cycles.}
	\label{tab:eqdim_caminosciclos}
\end{table}

\subsection{Johnson graphs}\label{johnson}

Johnson graphs are important because of their connections
with other combinatorial structures such as projective planes
and symmetric designs~\cite{b}. Furthermore, there exist different
studies about geometric versions of these graphs because of their multiple
 applications in network design (see for instance~\cite{Johnson_diam}).
 Due to these facts, among others, properties of Johnson graphs have been widely studied
  in the literature: spectra~\cite{Johnson_spectra}, induced subgraphs~\cite{Johnson_induced},
   connectivity~\cite{Johnson_connected}, colorings~\cite{Johnson_color}, distances~\cite{Johnson_distance},
   automorphisms~\cite{Johnson_aut} and
metric dimension~\cite{b}.
In this subsection we study the equidistant dimension of Johnson graphs, obtaining an upper bound for several cases.

The {\em Johnson graph} $J(n,k)$, with $n>k\ge 1$, has as vertex set the $k$-subsets of a $n$-set and two vertices are adjacent if their intersection
has size $k-1$. Thus, it can be easily seen that the distance between any two vertices $X,Y$ is given by
$$d(X,Y) = |X\setminus Y| = |Y\setminus X|=k-|X\cap Y|.$$
Consequently, a vertex $U\in V(J(n,k))$ is equidistant from vertices $X$ and $Y$ if and only if $|U\cap X| = |U\cap Y|.$

\begin{proposition} For any positive integer $k$, it holds that
$$eqdim(J(n,k))\leq n$$ whenever $n\in \{2k-1,2k+1\}$ or $n > 2k^2$.
\end{proposition}

\begin{proof}
Consider the vertices of $J(n,k)$ as $k$-subsets of the $n$-set $W=\{0,\ldots,n-1\}$.
For each positive integer $i$, let $S_i = \{i, i+1, \ldots, i+k-1\}\in V(J(n,k))$ where sums are taken modulo $n$ (thus $S_{i+rn}=S_i$ for any integer $r$).
We claim that the set ${\cal S} = \{S_0,\ldots,S_{n-1}\}$ is a distance-equalizer set of $J(n,k)$.
Suppose on the contrary the existence of two vertices $X,Y\in V(J(n,k))\setminus {\cal S}$ such that $|S_i\cap X| \neq |S_i\cap Y|$ for every  $i\in  \{0,1,\dots ,n-1\}$.

First, we assume that $n > 2k^2$.
For $j,r$ integers, let $T(j,r)=\{j,j+1,\ldots,j+r\}\subseteq  W$, where the sums are also taken module $n$.
Let $\mathcal{T}=\{T_1,\dots ,T_s\}$ be the family of sets $T(j,r)$ satisfying $T(j,r) \cap (X\cup Y)=\emptyset$, $j-1\in X\cup Y$ and $j+r+1\in X\cup Y$. Note that $\mathcal{T}$ is a partition of $W\setminus (X\cup Y)$ with at most $2k$ parts, by construction. Moreover,
$|T_i|\leq k-1$ for every $i\in \{1,\dots ,s\}$, otherwise, if $T_i=T(j,r)$ then $|S_j\cap X|=|S_j\cap Y|=0$, which contradicts our hypothesis.
Therefore,
$$n=|W|=|X\cup Y| + |W\setminus {(X\cup Y)}|\leq
2k+2k(k-1)=
2k^2,$$
contradicting our assumption
on $n$.

Now, suppose $n\in\{2k-1,2k+1\}$.  Let $u = (u_0,\ldots,u_{n-1})$ be the vector of
$\{-1,1,0\}^n$ such that $u_i=1$ if $i\in X\setminus Y$; $u_i=-1$
if $i\in Y\setminus X$; and $u_i=0$ otherwise.
Observe that $u$ has at most $2k$ non-zero components, and the same number of $1$'s and $-1$'s.  Hence, $\sum_{i=0}^{n-1}u_i=0$.
Let $s_i = \sum_{j=i}^{i+k-1}u_j$.
Observe that
$s_i = |S_i\cap X|-|S_i\cap Y|$,  for every $i\in \{0,1,\dots ,n-1\}$.
Hence,  $s_i\not= 0$, for every $i\in \{0,1,\dots ,n-1\}$,
because no set $S_i$ is equidistant from $X$ and $Y$.
Next, we prove that $s_i s_{i+k}<0$ for every $i\in \{0,\dots ,n-1\}$.
Indeed, we have that
\begin{align*}
s_i + s_{i+k} + u_{i+2k}=\sum_{i=0}^{n-1}u_i=0, \textrm{ when }n=2k+1;\\
s_i + s_{i+k} - u_{i}=\sum_{i=0}^{n-1}u_i=0, \textrm{ when }n=2k-1.
\end{align*}
Therefore, for $n=2k+1$,
\begin{align*}
s_{i+k} &= -s_i-u_{i+2k}\leq -1 - u_{i+2k}\le 0, \textrm{ if } s_i>0;\\
s_{i+k} &= -s_i-u_{i+2k}\ge 1-u_{i+2k} \ge 0,\textrm{ if } s_i<0
\end{align*}
and for $n=2k-1$,
\begin{align*}
s_{i+k} &= -s_i+u_{i}\leq -1 + u_{i}\le 0, \textrm{ if } s_i>0;\\
s_{i+k} &= -s_i+u_{i}\ge 1 + u_{i} \ge 0,\textrm{ if } s_i<0.
\end{align*}

Hence, for every $i\in \{0,\dots, n-1\}$, we have $s_is_{i+k}<0$ since $s_{i+k}\not= 0$, and it can be derived that  $s_is_{i+rk}<0$, for $r$ odd, and  $s_is_{i+rk}>0$, for $r$ even.
Then, $s_i s_{i+nk}<0$, since $n$ is odd, which is a contradiction since $s_{i+nk}=s_i$.
\end{proof}


\section{Using distance-equalizer sets for constructing doubly resolving sets}\label{doubly}

We now explore different relationships among distance-equalizer sets and doubly resolving sets. To do this, we first need to formally
define resolving sets.
 Indeed, a  subset $S$ of vertices  is a resolving set of a graph $G$ if,
 for every pair of vertices $x,y\in V(G)$, there exists a vertex $v\in S$ such that $d(v,x)\neq d(v,y)$; the metric dimension of $G$, denoted
  by $dim(G)$, is the minimum cardinality of a resolving set of $G$.
 Observe that, on the one hand, a set of vertices can be at the same time resolving set and distance-equalizer set. For example, it
  is easy to check that any independent set $S$ of cardinality  three of a cycle of order 6 is both resolving and distance-equalizer.
However, in this case $S$ is a distance-equalizer set of minimum cardinality,  but $S$ is not a resolving set of minimum cardinality, since
$eqdim(C_6)=3$ and $dim(C_6)=2$.
On the other hand, there are graphs satisfying $dim(G) = eqdim(G)$ with no minimum  resolving set being a
distance-equalizer set. For example, the cycle or order 4
satisfies $eqdim(C_4)=dim(C_4)=2$, but there is no set of cardinality two that is both resolving and distance-equalizer because
 a resolving set on two vertices is formed by two adjacent vertices and distance-equalizer sets of cardinality two are formed
  by two non-adjacent vertices.
\medskip

Doubly resolving sets were introduced in \cite{our} as a tool for computing the metric dimension of cartesian
products of graphs. Furthermore, different authors have provided
 interesting applications of doubly resolving sets on source location \cite{cw,j}, algorithmic studies and relations with other graph parameters \cite{kratica0,ghm,doubly, kratica2}.
 We say that two vertices $u, v$ \emph{doubly
 	resolve} a pair of vertices $x,y$ of $G$ (or that that $\{x, y\}$ are {\em doubly resolved} by $u, v$)
 if $d(u, x) - d(u, y) \neq d(v, x) - d(v, y)$. A set $S\subseteq V(G)$ is a \emph{doubly resolving set} of $G$ if every pair $\{x, y\} \subseteq V(G)$ is doubly resolved by two vertices of $S$ (it is said that $S$ {\it doubly resolves} $\{x, y\}$), and the minimum cardinality of such a set is denoted by $\psi(G)$. Observe that a doubly resolving set is also a resolving set, and so
 $dim(G) \leq\psi(G$).

\begin{proposition}
For every graph $G$, it holds that
$$\psi (G)\leq dim(G) +2\, \, eqdim (G).$$
\end{proposition}
\begin{proof}
 Let $A$ be a resolving set, and let $B$ be a distance-equalizer set.
We next construct a set $C$ of vertices satisfying $0\le |C|\leq eqdim(G)$ and such that for every pair
of vertices $x,y\in V(G)$, there exist $u,v\in A\cup B\cup C$ doubly resolving $x$ and $y$.

First, notice that if $x, y \in B$ then $x$ and $y$ are doubly resolved by themselves,
 and if $x, y  \notin B$, then $x$ and $y$ are doubly resolved by vertices $u$ and
$v$, where $u\in A$ is a vertex resolving $x$ and $y$, and $v\in B$ is equidistant from $x$ and $y$. Hence, in both cases,
$x$ and $y$ are doubly resolved by a pair of vertices  in $A\cup B$.

Now, suppose that $x\in B$ and $y\notin B$.
We claim that, for every $x\in B$ there is at most one vertex $y_x\notin B$ such that $x,y_x$ are not doubly resolved by $A\cup B$ and besides, in such a case, $d(u,x)+d(x,y_x)=d(u,y_x)$ for every $u\in A\cup B$.
Indeed, suppose that there exists $y'\notin B$ such that the pair $x, y'$ is not doubly resolved by $A\cup B$. Then, for all $u\in A\cup B$, the pair of vertices $x,u\in A\cup B$ does not doubly resolve $x,y'$. Hence, $d(u,x)-d(u,y')=-d(x,y')$.
In a similar way, if $y''$  is a vertex such that $y''\notin B$, $y''\not= y'$ and the pair $x, y''$ is not doubly resolved by $A\cup B$, then for all $u\in A\cup B$ we have $d(u,x)-d(u,y'')=-d(x,y'')$.
Therefore, $d(x,y')-d(x,y'')=d(u,y')-d(u,y'')$. Thus, for every pair of vertices $u,v\in A\cup B$ we obtain
$d(u,y')-d(u,y'')=d(x,y')-d(x,y'')=d(v,y')-d(v,y'')$, implying that $y',y''\notin B$ are not doubly resolved by $A\cup B$, which is not possible as we have seen in the former paragraph.
Consider the (possibly empty) set

$$C=\{ y_x  : x\in B,\,\,  y_x \notin B \text{ and } x, y_x\text{ are not doubly resolved by }A\cup B \} .$$

Then, $0\le |C|\le |B|$ and, by construction, $A\cup B\cup C$ doubly resolves $x$ and $y$ whenever  $x\in B$ and $y\notin B$.
Therefore, the set  $S=A\cup B \cup C$ is a doubly resolving set for $G$ and, consequently, $|S|\leq dim(G) +2\,\, eqdim (G)$.
\end{proof}

We think that the preceding bound can be improved as follows.

\vspace{0.3cm}
\noindent
\begin{conjecture}\label{conjpsi}
For every graph $G$, it holds that $\psi(G)\leq dim(G)+ eqdim(G)$.
\end{conjecture}
\vspace{0.3cm}

A graph attaining the upper bound given in the preceding conjecture is shown in Figure \ref{conjecture}.

\begin{figure}[!t]
	\begin{center}
		\includegraphics[width=0.22\textwidth]{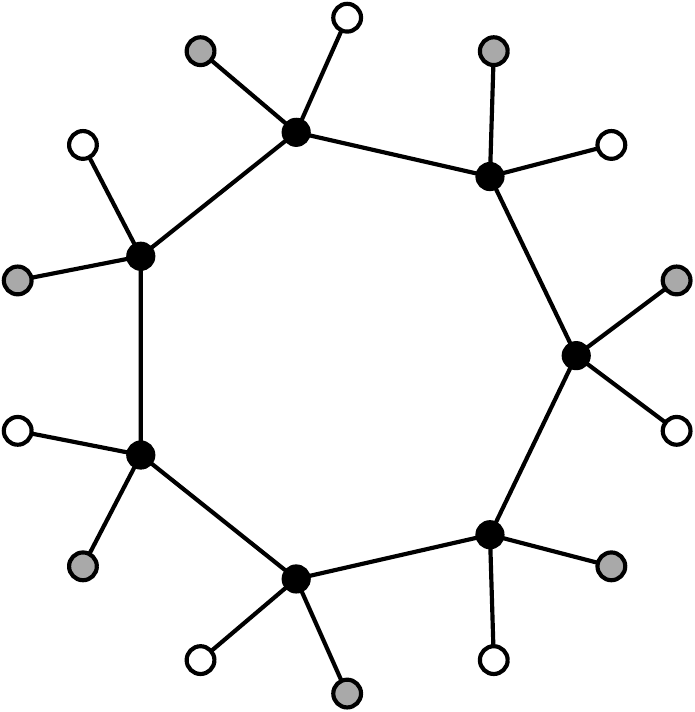}
		\caption{ In this graph $G$, $\psi(G)=dim(G)+eqdim(G)$. Black vertices, gray vertices and the set
of leaves are a distance-equalizer set, a resolving set and a doubly resolving set of minimum cardinality, respectively.}
		\label{conjecture}
	\end{center}
\end{figure}

Although we have no proof of Conjecture~\ref{conjpsi}, we next prove that it holds true for trees.

\begin{theorem}
	For every tree $T$, it holds that $$\psi(T)\leq dim(T)+ eqdim(T).$$
\end{theorem}

\begin{proof}
    Let $S$ be the set of all support vertices of $T$. For every $v \in S$, consider the sets
	of vertices $L_v =\{z : \ z  \ \hbox{is a leaf adjacent to} \ v\}$ and
$ T_v =\{v\}\cup L_v$, and
	observe that the sets $T_v $ are pairwise disjoint. 	
	First, note that it is well-known that the set of leaves of a tree $T$ is the unique minimum doubly resolving set of $T$ (see \cite{our}). Hence, $$\psi(T)=\displaystyle\sum_{v\in S} |L_v|.$$
		Also, note that if $W$ is a resolving set of $T$, then $| W \cap L_v|\geq |L_v|-1$ for every  $v\in S$ (see \cite{ours2}).
	
	Let $W'$ be the union of a minimum  resolving set and a minimum  distance-equalizer set of $T$.
	We claim that $|W' \cap T_v|\geq |L_v|$, for every  $v\in S$.
	Indeed,  $|W'\cap L_v|\geq |L_v|-1$, since $W'$ is a resolving set, and  $v \in W'$ or $L_v\subseteq W'$,
	by Lemma~\ref{hoja-soporte}. In any case, $|W' \cap T_v|\geq |L_v|$.
	Then,
	
	$$\psi (T)= \sum_{v\in S} |L_v|\le \sum_{v\in S} |W' \cap T_v|\le |W'|\le dim(T)+eqdim (T).$$
\end{proof}

We finish this section analyzing lower and upper bounds on   $dim(G)+eqdim(G)$. Concretely, we are interested in
the minimum and maximum value of $dim(G)+eqdim(G)$ for graphs of order $n$.
First, note that for any nontrivial graph $G$ of order $n$,
\begin{equation}\label{cotadimeqdim}2\le dim(G)+eqdim(G)\le 2(n-1).\end{equation}
The lower bound in (\ref{cotadimeqdim}) is attained only by the paths $P_2$ and $P_3$, by Theorem~\ref{extremalinferior}(i),
 and the upper bound, only by the path $P_2$, by Theorem
 ~\ref{extremalsuperior}(i).
Hence, for every graph $G$ of order at least $4$,
\begin{equation*}\label{cotadimeqdim4}3\le dim(G)+eqdim(G)\le 2n-3.\end{equation*}
In order to study this question, we consider the following functions defined for integers $n\geq 4$:
\begin{align*}
\Sigma(n)&:=\displaystyle\max \{dim(G)+eqdim(G) \ : \ |V(G)|=n \}\\
\sigma(n)&:=\displaystyle\min \{dim(G)+eqdim(G) \ : \ |V(G)|=n \}.
\end{align*}

\begin{proposition} For every  integer $n\ge 4$, the following statements hold.
	\begin{enumerate}[i)]
		\item $\Sigma(n)\geq \frac{3n}{2}-3$;
		\item $\sigma(n)\leq \log_2(n)+2$.
	\end{enumerate}
\end{proposition}
\begin{proof}
	\begin{enumerate}[$i)$]
		\item It is enough to consider the complete bipartite graph $G=K_{\lfloor n/2\rfloor,\lceil n/2\rceil}$, for which
 $dim(G)=n-2$ (see~\cite{CEJO-DAM00}), and $eqdim(G)=\lfloor n/2\rfloor$, by Theorem~\ref{bipartitoscompletos}(ii).
Hence, $dim (G)+eqdim(G)= n-2 + \lfloor n/2\rfloor\geq 3n/2-3$.
		
		\item For every $k\ge 1$, consider the graph $G_k$, with  $V(G_k)=A\cup B\cup C$, where
		$ \ A=\{v\}, \ B=\{1, \dots, k\}$, $ \ \
		C=\{w : \ w \ \ \hbox{is a binary word of length k}  \}$ and two different vertices $x$ and $y$ are adjacent in $G_k$
 if and only if one of the following conditions hold (see an example in Figure~\ref{g3}):
	    \begin{enumerate}[$\bullet$]
	    \item one of the vertices is $v$;
	
	    \item one of the vertices belongs to $C$, say $x=w\in C$, and the other one belongs to $B$, say  $y=j\in B$, and
	    $w$ has the digit $1$ in the $j$-th position.
	    \end{enumerate}
	     Then, $G_k$ is a graph of order $n=2^k+k+1$. Moreover, $A$  is a
 distance-equalizer set since $v$ is a universal vertex, and it is easy to check that $A\cup B$ is a resolving set. Hence,  $dim (G_k)+eqdim(G_k)\le k+2\le  \log_2(n)+2$.
	\end{enumerate}
\end{proof}
\begin{figure}[!t]
	\begin{center}
		\includegraphics[width=0.5\textwidth]{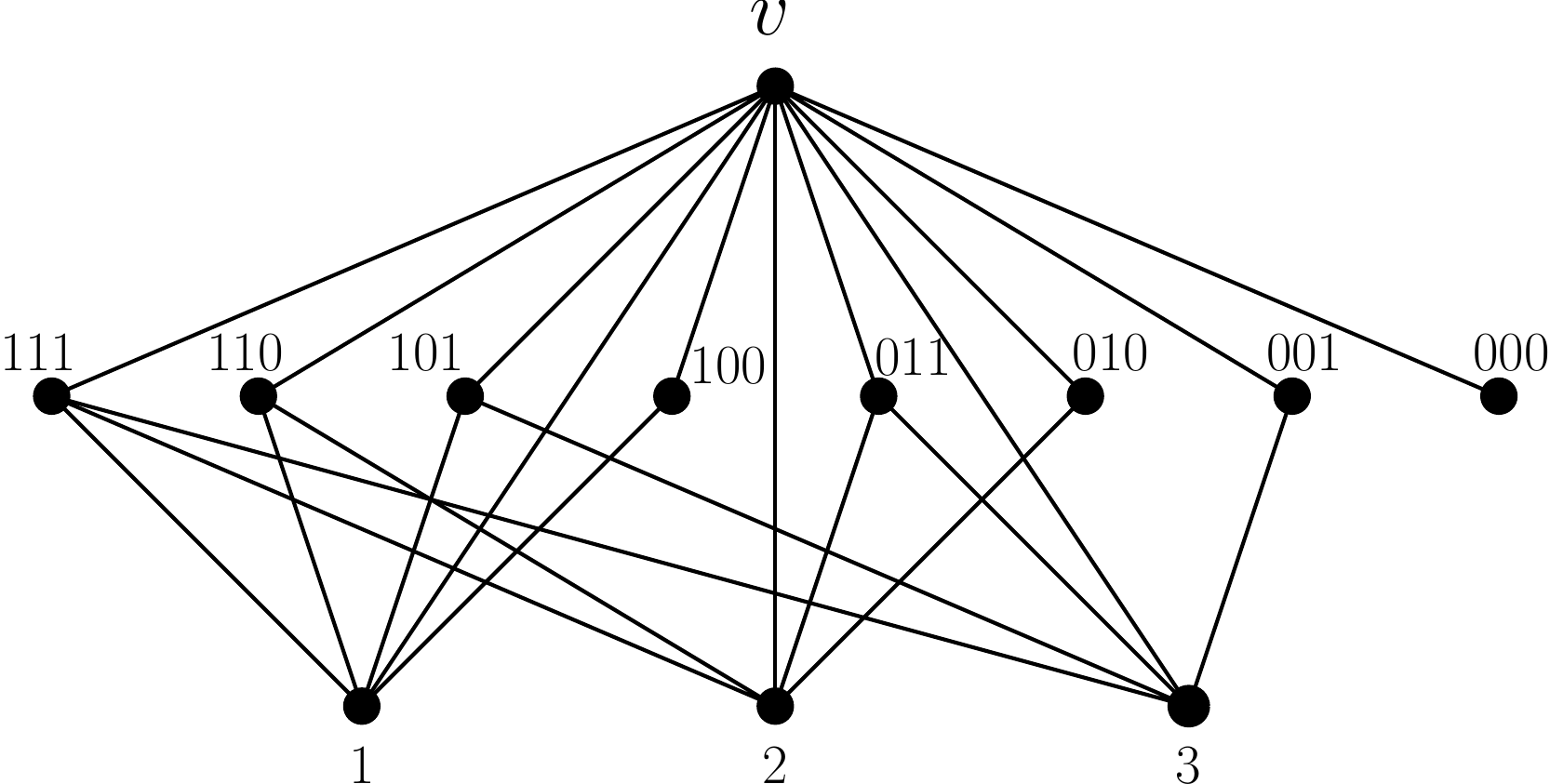}
		\caption{ The graph $G_3$.}
		\label{g3}
	\end{center}
\end{figure}

\section{Conclusions and open problems}\label{op}

In this paper, the notion of equidistant dimension as a parameter to evaluate sameness in graphs is introduced.
The value of this invariant in several families of graphs
and  relations with other parameters
 have been provided. In Table 2, the equidistant dimension, metric dimension and minimum cardinality of doubly resolving
 sets of some families of graphs are given.
Also, all graphs reaching some extremal values of the equidistant dimension have been characterized.

\begin{table}[!t]
\label{basicfamilies}
\center{
\begin{tabular}{p{1.5cm}p{4cm}p{4.5cm}p{1.3cm}p{1cm}}
  \hline
  $G$ & $constraints$ & $eqdim(G)$ & $dim(G)$ & $\psi(G)$ \\  \hline \hline
  $P_n$ & $n\geq 2$ & $n-r \left( \Big \lceil {{\frac{n}{2}}} \Big \rceil \right)$ & 1 & 2 \\  \hline
   & $n=4k\geq 4$ & $\frac{3n}{4}-1$ & 2 & 3 \\
  $C_n$ & $n=4k+2\geq 6$ & $\frac{n}{2}$ & 2 & 3 \\
  & $n=2k+1\geq 5$ & $\leq n-r\left (\Big \lceil\frac{n+1}{4} \Big \rceil  \right)$ & 2 & 2 \\  \hline
  $K_n$ & $n\geq 3$ & 1 & $n-1$ & $n-1$ \\  \hline
  $K_{r,s}$ & $2\leq r\leq s=n-r$ & r & $n-2$ & $n-2$ \\  \hline
  $K_{1,n-1}$ & $n\geq 4$ & 1 & $n-2$ & $n-1$ \\  \hline
  $K_2(r,s)$ & $3\leq r\leq s=n-r$ & r & $n-4$ & $n-2$ \\  \hline
  $K_{n_1,\dots,n_p}$ & $p\geq 3$, $n_1+\dots +n_p=n $ & min$\{3,n_1,\dots,n_p\}$ & $n-p$ & $n-p$  \\
      \hline
\end{tabular}
\caption{Equidistant dimension and related parameters of some families of graphs.}}
\end{table}

As future work, besides solving Conjecture~\ref{conjtrees} about the equidistant dimension of trees, and Conjecture~\ref{conjpsi}
about the relation of the equidistant dimension with doubly resolving sets, it would be interesting to relate distance-equalizer
sets to other types of sets of vertices such as dominating sets, cut sets or determining sets, for example. Also, it could be of
interest to find other graph families whose equidistant dimension connects with other problems, thus producing similar results as
the relationship between the computation of this parameter in paths and AP-3-free sequences. Finally, we could perform new
techniques that allow us to compute exact values for the equidistant dimension of Johnson graphs.


\begin{thebibliography}{10}\label{bibliography}

\bibitem{Johnson_connected}
B. Alspach,
Johnson graphs are Hamilton-connected, Ars Mathematica Contemporanea 6 (2013) 21--23.

\bibitem{Johnson_induced}
M. Aslam, A. Ali,
Some results on induced subgraphs of Johnson graphs, Int. Math. Forum 7 (9-12) (2012)  445--454.

\bibitem{b}
R. F. Bailey, J. C\'{a}ceres, D. Garijo, A. Gonz\'{a}lez, A. M\'{a}rquez, K. Meagher, M. L. Puertas,
Resolving sets for Johnson and Kneser graphs,
European Journal of Combinatorics 34 (4) (2013) 736--751.

\bibitem{Johnson_diam}
 C. Bautista-Santiago, J. Cano, R. Fabila-Monroy, D. Flores-Peñaloza, H. Gonz\'{a}lez-Aguilar, D. Lara, E. Sarmiento, J. Urrutia,
 On the connectedness and diameter of a geometric Johnson graph,
 Discrete Mathematics and Theoretical Computer Science 15 (3) (2013) 21--30.

\bibitem{b2}
F. A. Behrend,
On sets of integers which contain no three in arithmetic progression,
Proc. Nat. Acad. Sci. 23 (1946) 331--332.

\bibitem{Johnson_color}
S. Bitan, T. Etzion,
On the chromatic number, colorings, and codes of the Johnson graph, Discrete Appl. Math. 70 (2) (1996) 163--175.

\bibitem{*bloom16}
T. F. Bloom,
A quantitative improvement for Roth's theorem on arithmetic progressions,
Journal of the London Mathematical Society, Second Series, 93 (3) (2016) 643--663.

\bibitem{*bloom19}
T. F. Bloom, O. Sisask,
 Logarithmic bounds for Roth's theorem via almost-periodicity,
 Discrete Analysis 4 (2019) 20 pp.

\bibitem{*bourgain}
J. Bourgain, On triples in arithmetic progression, Geometric and Functional Analysis 9 (5) (1999) 968--984.

\bibitem{ourinf}
 J. C\'{a}ceres, C. Hernando, M. Mora, I. M. Pelayo, M. L.  Puertas, C. Seara,
On the metric dimension of infinite graphs,
Discrete Applied Mathematics 160 (18)(2012) 2618--2626.

\bibitem{our} J. C\'{a}ceres, C. Hernando, M. Mora, I. M. Pelayo, M. L.  Puertas, C. Seara, D. R. Wood,
On the metric dimension of Cartesian products of graphs,
SIAM J. Discrete Math. 21(2) (2007) 423--441.

\bibitem{kratica0}
M. \v{C}angalovi\'{c}, J. Kratica, V. K. Vuj\v{c}i\v{c}, M. Stojanovi\v{c},
Minimal double resolving sets of prism graphs,
Optimization 62 (8)(2013) 1037--1043.

\bibitem{c} J. Casas-Roma, J. Herrera-Joancomartí, V. Torra,
A survey of graph-modification techniques for privacy-preserving on networks,
Artificial Intelligence Review 47 (3) (2017) 341--366.

\bibitem{CEJO-DAM00}
{ G. Chartrand, L. Eroh, M. A. Johnson, O. R. Oellermann},
{Resolvability in graphs and the metric dimension of a graph},
Discrete Appl. Math. 105 (2000) 99--113.

\bibitem{cw} X. Chen,  C. Wang,
Approximability of the minimum weighted doubly resolving set problem,
Computing and Combinatorics, LNCS 8591 (2014) 357--368.

\bibitem{cs} S. Chester, G. Srivastava,
Social network privacy for attribute disclosure attacks,
2011 International Conference on Advances in Social Networks Analysis and Mining, IEEE (2011) 445--449.

\bibitem{cs2} S. Chester, B. M. Kapron, G. Srivastava, S. Venkatesh,
Complexity of social network anonymization,
Social Network Analysis and Mining 3 (2) (2013) 151--166.

\bibitem{CGHHMMT21} M. Claverol, A. García, G. Hernández, C. Hernando, M. Maureso, M. Mora, J. Tejel,
Metric dimension of maximal outerplanar graphs, 
Bulletin of the Malaysian Mathematical Sciences Society 44 (4) (2021)  2603--2630.

\bibitem{cockaine}
E. J. Cockayne, S. T. Hedetniemi,  On the diagonal queens domination problem, Journal of Combinatorial Theory, Series A, 42 (1) (1986) 137--139.

\bibitem{Johnson_distance}
A. Dabrowski, L. S. Moss,
 The Johnson graphs satisfy a distance extension property, Combinatorica 20 (2) (2000) 295--300.

\bibitem{sequences}
J. Dybizba\'{n}ski,
Sequences containing no 3-term arithmetic progressions,
The Electron. J. of Combin. 19 (2) (2012) P15.

\bibitem{*elkin}
M. Elkin, An improved construction of progression-free sets, Israel Journal of Mathematics 184 (2011) 93--128.

\bibitem{erdos} P. Erd\"{o}s, P. Tur\'{a}n,
On some sequences of integers,
J. London Math. Soc. 11 (1936) 261--264.

\bibitem{anonimo}
 T. Feder, S. Nabar, E. Terzi,
 Anonymizing graphs,
\emph{CoRR}, abs/0810.5578 (2008).

\bibitem{*gasarch}
W. Gasarch, J. Glenn, C. Kruskal,
 Finding large 3-free sets. I. The small n case,
 Journal of Computer and System Sciences, 74 (4) (2008) 628--655.

\bibitem{ghm}
A. Gonz\'{a}lez, C. Hernando, M. Mora,
Metric-locating-dominating sets of graphs for constructing related subsets of vertices,
Appl. Math. Comput. 332 (2018) 449--456.

\bibitem{*gowers}
W. Gowers, A new proof of Szeméredi's theorem, Geom. Funct. Anal. 11 (2001) 465--588.

\bibitem{harary} F. Harary, R. A. Melter,
On the metric dimension of a graph,
Ars Combin. 2 (1976) 191--195.

\bibitem{hb}
D. R. Heath-Brown,
Integer sets containing no arithmetic progressions,
J.London Math. Soc. (2) 35 (1987), 385--394.

\bibitem{ours2}
C. Hernando, M. Mora, I. M. Pelayo, C. Seara, D. R. Wood,
Extremal graph theory for metric dimension and diameter,
Electron. J. Combin. 17 (2010) R30.

  \bibitem{j}
 J. Jiang, S. Wen, S. Yu, Y. Xiang,  W. Zhou,
 Identifying propagation sources in networks:State-of-the-art and comparative studies,
 IEEE Communications Surveys Tutorials 19 (1)  (2017) 465--481.

 \bibitem{doubly}
J.~Kratica, M.~{\v{C}}angalovi{\'c},
  V.~Kova{\v{c}}evi{\'c}-Vuj{\v{c}}i{\'c},
\newblock Computing minimal doubly resolving sets of graphs,
\newblock {\em Comput. Oper. Res.} 36 (7) (2009) 2149--2159.

\bibitem{kratica2}
J.~Kratica, V.~Kova{\v{c}}evi{\'c}-Vuj{\v{c}}i{\'c}, M.~{\v{C}}angalovi{\'c},
  M.~Stojanovi{\'c},
\newblock Minimal doubly resolving sets and the strong metric dimension of some
  convex polytopes,
\newblock {\em Appl. Math. Comput.} 218 (19) (2012) 9790--9801.

\bibitem{Johnson_spectra}
M. Krebs, A. Shaheen,
 On the spectra of Johnson graphs, Electron. J. Linear Algebra 17 (2008) 154--167.

\bibitem{lobstein}
A. Lobstein, Watching systems, identifying, locating-dominating and discriminating codes in graphs,
 http://www.infres.enst.fr/lobstein/debutBIBidetlocdom.pdf

\bibitem{Johnson_aut}
M. Ramras, E. Donovan,
 The Automorphism Group of a Johnson Graph, SIAM J. Discrete Math. 25 (2011) 267--270.

\bibitem{roth}
K.F. Roth, On certain sets of integers,
J. London Math. Soc. 28 (1) (1953) 104--109.

\bibitem{*sanders}
T. Sanders, On Roth's theorem on progressions, Annals of Mathematics, Second Series 174 (1) (2011) 619--636.

\bibitem{sharma}
A. Sharma, Sequences of Integers Avoiding 3-term Arithmetic Progressions, Elec. J. of Comb. 19 (2012) P27.


\bibitem{slater} P. J. Slater,
Leaves of trees,
Congressus Numerantium 14 (1975) 549--559.

\bibitem{yero} R. Trujillo-Rasua, I. G. Yero,
k-Metric Antidimension: a Privacy Measure for Social Graphs,
Information Sciences 328 (2016) 403--417.

\bibitem{west}
D. B. West,
Introduction to Graph Theory, Second Edition,
Prentice Hall, 2001.

\bibitem{z}
E. Zheleva, L. Getoor,  Privacy in Social Networks: A Survey. In: C. C. Aggarwal(ed) Social Network Data Analytics, 1st edn., Springer (2011) 277--306.

\bibitem{z2}
B. Zhou, J. Pei, W. S Luk,
A brief survey on anonymization techniques for privacy preserving publishing of social network data
SIGKDD Explor. Newsl. 10 (2) (2008) 12--22.

\end{thebibliography}
\end{document}